\documentclass[a4paper,12pt]{article}
\usepackage{a4wide}
\usepackage{amsmath}
\usepackage{amssymb}
\usepackage{amsthm}
\usepackage{latexsym}
\usepackage{graphicx}
\usepackage[english]{babel}
\usepackage{makeidx}
\usepackage{mathtools}

\newtheorem{obs} [subsection]{Remark}

\newtheorem{prop}[subsection]{Proposition}

\newtheorem{teor}[subsection]{Theorem}
\newtheorem{lema}[subsection]{Lemma}
\newtheorem{cor} [subsection]{Corollary}

\newcommand{\Be}{\mathcal B_{k}}
\newcommand{\Ie}{\mathcal I_{k}}

\newcommand{\Bep}{\mathcal B_{k'}}

\newcommand{\Ok}{\mathcal O_k}

\newcommand{\ONR}{\Omega(\mathbb N,R)}
\newcommand{\LL}{L}
\newcommand{\NN}{\mathbb N}

\def\gcd{\operatorname{gcd}}

\def\Ree{\operatorname{Re}}
\numberwithin{equation}{section}
\usepackage[pagewise]{lineno}  %\linenumbers
\begin{document}
\selectlanguage{english}
\frenchspacing

\large
\begin{center}
\textbf{A note on the linear independence of a class of series of functions}

Mircea Cimpoea\c s
\end{center}
\normalsize

\begin{abstract}\small{
For $k\in\mathbb R$, we consider a $\mathbb C$-algebra $\mathcal A_k$ of holomorphic functions in the half plane $\Ree z>k$ 
with (at most) subexponential growth on the real line to $+\infty$. In the $\mathcal A_k$-algebra of sequences of functions 
$\{\alpha:\NN\rightarrow \mathcal A_k\}$, we consider the $\mathcal A_k$-subalgebra $\mathcal H_k$ consisting in those $\alpha$ 
for which there exists a continuous map $M:\{\Ree z>k\}\rightarrow [0,+\infty)$ such that $|\alpha(n)(z)|\leq M(z)n^k$ for all $\Ree z>k,n\geq 1$, and 
$\lim_{x\rightarrow +\infty}e^{-ax}M(x)=0$, for all $a>0$.
Given $L$ a sequence of holomorphic functions on $\Ree z>k$ which satisfies certain conditions, we prove that the map $\alpha\mapsto F_L(\alpha)$, 
where $F_L(\alpha):=\sum_{n=1}^{+\infty}\alpha(n)(z)L(n)(z)$, is an injective morphism of $\mathcal A_k$-modules (or $\mathcal A_k$-algebras). 
Consequently, if $n\mapsto \alpha_j(n)(z)\in\mathbb C$, $1\leq j\leq r$, are linearly (algebraically) independent over $\mathbb C$, 
for $z$ in a nondiscrete subset of $\Ree z>k$, then $F_{\alpha_1},\ldots,F_{\alpha_r}$ are linearly (algebraically) independent 
over the quotient field of $\mathcal A_k$.

\noindent \textbf{Keywords:} series of functions; meromorphic functions; Dirichlet series.

\noindent \textbf{2010 MSC:} 30B50; 30D30}
\end{abstract}

\section*{Introduction}

Let $R$ be a commutative ring with unity and let $\NN=\{1,2,3,\ldots\}$ be the set of positive integers. The set of functions
$\Omega(\NN,R):=\{\alpha:\NN \rightarrow R\}$ has a natural structure of a $R$-algebra, with the operations:
\begin{eqnarray*}
& (\alpha+\beta)(n):=\alpha(n)+\beta(n),\;\forall n\in \NN, \\
& (\alpha \cdot \beta)(n):=\sum_{ab=n}\alpha(a)\beta(b),\;\forall n\in \NN.
\end{eqnarray*}
Moreover, if $R$ is a domain, then $\Omega(\NN,R)$ is also. Assume $R=\mathbb C$. 
The algebraic properties of the ring $\Omega(\NN,\mathbb C)$, known as the Dirichlet ring or the ring of arithmetic functions, were intensively studied in the literature, see for instance \cite{berb}, \cite{cash} and \cite{cash2}.

Let $k\in\mathbb R$. We denote $\mathcal O_k$ the domain of holomorphic functions on $\Ree z>k$. We let
$$O(n^k)=\{ \alpha\in\Omega(\NN,\mathbb C)\;|\;\exists C>0, \text{ such that} |\alpha(n)|\leq Cn^{k},\;\forall n\geq 1 \},$$
the set of arithmetic functions of order $n^k$. We denote $\mathcal D_k:=\bigcap_{\varepsilon>0}O(n^{k+\epsilon})$. It is easy to check
that $\mathcal D_k$ is a $\mathbb C$-subalgebra of $\Omega(\NN,\mathbb C)$. For any $\alpha\in \mathcal D_k$,
the \emph{Dirichlet series} $$F(\alpha)(z):=\sum_{n=1}^{+\infty}\frac{\alpha(n)}{n^z},$$
uniformly absolutely converge on the compact subsets of $\Ree z> k+1$, hence $F(\alpha)\in \mathcal O_{k+1}$. 
Also, $F(\alpha)$ is identically zero if and only if $\alpha=0$ (see \cite[Theorem 11.3]{apostol}). Therefore, 
by straightforward computations, the map
$F:\mathcal D_k \rightarrow \mathcal O_{k+1},\; \alpha\mapsto F(\alpha)$,
is an injective morphism of $\mathbb C$-algebras. 

Consequently, if $\alpha_1,\ldots,\alpha_r\in\mathcal D_k$ are linearly independent (algebraically independent) over $\mathbb C$,
then $F(\alpha_1),\ldots,F(\alpha_r)$ are also linearly independent (algebraically independent) over $\mathbb C$. 
This remark has very important consequences in analytic number theory, see for instance \cite{florin}, \cite{kac} and \cite{cim}.
The aim of our paper is to study the association $\alpha\mapsto F(\alpha)$ in a larger context. Our approach follows the methods used in
\cite{cim}. 

Let $k\in \mathbb R$. We consider the subsets 
\begin{eqnarray*}
& \mathcal C_k:=\{\alpha\in \Omega(\NN, \mathcal O_k) :\forall \varepsilon>0, \exists M_{\varepsilon}:\{\Ree z>k\}\rightarrow [0,+\infty) \text{ continuous }\\
& \text{ such that } |\alpha(n)(z)|<n^{k+\varepsilon}M_{\varepsilon}(z),\forall n\geq 1,\Ree z>k \},\\
& \mathcal E_k:=\{\LL\in \Omega(\NN,\mathcal O_{k})\;:\; \exists c>0,\;C:\{\Ree z>k\}\rightarrow [0,+\infty) \text{ continuous }\\
& \text{ such that } |\LL(n)(z)| \leq C(z)n^{- c(\Ree z)}, \;\forall n\geq 1, \Ree z>k\}.
\end{eqnarray*}
We prove that $\mathcal C_k$  has a natural structure of $\mathcal O_k$-algebra, see Proposition $1.1$.
In Proposition $1.4$, given $\alpha\in\mathcal D_k$ and $L\in\mathcal E_k$, we prove that there exists $k'\geq k$ such that the series of functions
$$F_L(\alpha)(z):=\sum_{n=1}^{+\infty}\alpha(n)(z)\LL(n)(z)$$
is uniformly absolutely convergent on the compact subsets of $\Ree z>k'$, hence $F(\alpha)\in \mathcal O_{k'}$.
In the main result of the first section, Theorem $1.5$, we prove that the map $\alpha\mapsto F_L(\alpha)$ is $\mathcal O_k$-linear and, moreover, is a morphism of
$\mathcal O_k$-algebras, if $L:(\NN,\cdot)\rightarrow (\mathcal O_k,\cdot)$ is a monoid morphism.

In the $\mathbb C$-vector space $\mathcal E_k$ we consider the subset
\begin{eqnarray*}
& \widetilde{\mathcal E}_k:=\{\LL\in \mathcal E_k\;:\; \LL(n)(z)\neq 0,\;\forall n\geq 1,\;\Ree z>k \text{ and } \\
& \forall n_0\geq 1,\; \exists  C(n_0)>0, \text{ such that } \frac{|L(n)(z)|}{|L(n_0)(z)|} \leq n^{-C(n_0)\Ree z},
\;\forall n\geq n_0+1,\;\Ree z>k \}.
\end{eqnarray*}
In Remark $1.6$, we note that the (general) Dirichlet series, see \cite{hardy}, and the classical Dirichlet series, see \cite{apostol}, are
particular cases of the type $F_{L}(\alpha)$, where $\alpha\in \Omega(N^*,\mathbb C)$ with polynomial growth and $L\in \widetilde{\mathcal E}_0$ has a special form. 

In the beginning of the second section, similarly to \cite{cim}, we define
\begin{eqnarray*}
& \Be := \{f \in \Ok \;:\;\forall a>0,\; \lim_{x\rightarrow +\infty}e^{-ax}|f(x)|=0 \}, \\
& \Ie := \{f\in \Ok \;:\exists a>0 \text{ such that} \;\lim_{x \rightarrow +\infty}e^{ax}|f(x)|=0 \},
\end{eqnarray*}
where the limits are taken on the real line.
In Proposition $2.1$ we show that $\Be$ is a subdomain in $\Ok$ and $\Ie$ is an ideal in $\Be$.
Moreover, in Proposition $2.5$ we prove that $\Ie$ does not contain non-zero entire functions of order $<1$.

Let $\mathcal A_{k} \subset (\mathcal B_{k}\setminus \mathcal I_{k})\cup \{0\}$ be a 
$\mathbb C$-subalgebra of $\mathcal B_{k}$. We consider the subset 
\begin{eqnarray*}
& \mathcal H_{k}:=\{ \alpha\in \Omega(\NN,\mathcal A_{k})\;:\;\exists M:\{\Ree z> k\}\rightarrow [0,+\infty) \text{ continuous, such that } \\
& \text{(i)} |\alpha(n)(z)|\leq M(z)n^{k},\;\forall n\geq 1,\Ree z>k,\text{ (ii) }\lim_{x\rightarrow+\infty} e^{-ax}M(x) = 0,\;\forall a>0\},
\end{eqnarray*}
which is an $\mathcal A_{k}$-subalgebra of $\Omega(\NN,\mathcal A_{k})$.

The main result of our paper is Theorem $2.7$, were we prove that for any $L\in \widetilde{\mathcal E}_k$, there exists a constant $k'\geq k$ which
depends on $L$, such that the map $F_L:\mathcal H_k \rightarrow \mathcal O_{k'}$, $\alpha\mapsto F_L(\alpha)$, is an injective morphism of $\mathcal A_k$-modules.
Moreover, if  $L:(\NN,\cdot)\rightarrow (\mathcal O_k,\cdot)$ is a monoid morphism, then $F_L$ is an injective morphism of $\mathcal A_k$-algebras.
Note that, Theorem $2.7$, in light of Remark $1.6$, generalize the identity theorem for (general) Dirichlet series, see \cite[Theorem 6]{hardy}.

Let $\alpha_1,\ldots,\alpha_r \in \mathcal H_{k}$ and assume there exists a nondiscrete subset $S\subset \{\Ree z\gg 0\}$ such that
the numerical sequences $n\mapsto \alpha_j(n)(z)$, $1\leq j\leq r$, are linearly independent over $\mathbb C$, for any $z\in S$.
Let $L\in\widetilde{\mathcal E}_{k}$. In Corollary $2.8$ we prove that $F_L(\alpha_1),\ldots,F_L(\alpha_r)\in \mathcal O_{k'}$ are linearly independent over 
$\mathcal F_{k'}:=$ the quotient field of $\mathcal A_{k'}$. Also, if $L$ is a monoid morphism between $(\NN,\cdot)$ and $(\mathcal O_k,\cdot)$ and 
 $n\mapsto \alpha_j(n)(z)$, $1\leq j\leq r$, are algebraically independent over $\mathbb C$, then $F_L(\alpha_1),\ldots,F_L(\alpha_r)$ are algebraically
independent over $\mathcal A_k$. The case of (general) Dirichlet series, discussed in Remark $1.6$, is reobtained as a particular case of Corollary $2.9$.

In the third section, we give an application to Dirichlet series associated to multiplicative arithmetic functions.  Given $\alpha_1,\ldots,\alpha_r\in \mathcal D_k$ multiplicative functions 
and an integer $m\geq 0$, such that $e,\alpha_1,\ldots,\alpha_r$ are pairwise non equivalent, in the sense of \cite{kac}, we prove that the derivatives of order $\leq m$, $F^{(i)}(\alpha_j)$, $1\leq j\leq r$,
$0\leq i\leq m$, are linearly independent over $\mathcal F_k$, and, in particular, over the field of  meromorphic functions of order $<1$, see Proposition $3.2$. 
This generalize the main result of \cite{kac}. Moreover, if $\alpha_1,\ldots,\alpha_r$ are algebraically independent over $\mathbb C$, then $F(\alpha_1),\ldots,F(\alpha_r)$ are algebraically
independent, see Proposition $3.3$. We also note in Remark $3.4$ the connections with the theory of Artin L-functions, see \cite{artin1}, and the main results of \cite{florin} and \cite{cim}. We note that other, and independent, cases of independence of suitable families
of Dirichlet series are proved in \cite{molt}.

\section{Algebras of sequences of holomorphic functions}

Let $R$ be a commutative ring with unity and denote $\NN$ the set of positive integers. In the set of functions
$\Omega(\NN,R):=\{\alpha:\NN \rightarrow R\}$, we consider two operations
\begin{eqnarray*}
& (\alpha+\beta)(n):=\alpha(n)+\beta(n),\;\forall n\in \NN, \\
& (\alpha \cdot \beta)(n):=\sum_{ab=n}\alpha(a)\beta(b),\;\forall n\in \NN.
\end{eqnarray*}
We denote $0,e \in \ONR$ the functions $0(n)=0$ for all $n\geq 1$, $e(1)=1$, $e(n)=0$, for all $n\geq 2$.
It is well known, see for instance \cite{cash} and \cite{cash2}, that $(\ONR,+,\cdot)$ is a commutative ring with the unity $e$.
Moreover, if $R$ is a domain then $\ONR$ is also a domain.

Let $\mathcal O:=\mathcal O(\mathbb C)$ be the domain of holomorphic (entire) functions, with the usual operations of addition and multiplication of functions. 
For any real number $k$, let $\mathcal O_k:=\mathcal O(\{\Ree z>k\})$ be the domain of holomorphic functions defined on the open half plane $\Ree z>k$.
Note that, the natural map 
$$i_k:\mathcal O \rightarrow \mathcal O_k, \; f\mapsto f|_{\Ree z>k},$$ is an injective morphism of $\mathbb C$-algebras. 
Indeed, if $f,g\in \mathcal O$ such that $f|_{\Ree z>k}=g|_{\Ree z>k}$, then, by the identity theorem of holomorphic functions, $f=g$.
If we see the maps $i_k$'s as inclusions, then $\mathcal O = \bigcap_{k\in\mathbb R} \mathcal O_k$. 
If $k\leq k'$ then the natural map $$i_{k,k'}:\mathcal O_{k} \rightarrow \mathcal O_{k'},\; f\mapsto f|_{\Ree z>k'},$$
is an injective morphism of $\mathbb C$-algebras. 

Moreover, $i_{k,k}$ is the identity map on $\mathcal O_k$ and for any $k''<k'<k$,
we have $i_{k,k''}=i_{k',k''}\circ i_{k,k'}$. Hence, $(\{\mathcal O_k\}_{k\in\mathbb R}, \{i_{k,k'}\}_{k\leq k'})$ is a direct system.
We denote $$\mathcal O_{\infty}:=\lim_{\longrightarrow}\mathcal O_k = \bigcup_{k\in\mathbb R}\mathcal O_k,$$
the direct limit of the above system. For the last equality, we see the maps $i_{k,k'}$'s as inclusions. Therefore 
$ \mathcal O \subset \mathcal O_k \subset \mathcal O_{\infty}$, for all $k\in \mathbb R$.
On the other hand, for any $k,k'\in\mathbb R$,
the map 
\begin{equation}\label{tkk}
T_{k-k'}:\mathcal O_{k'}\rightarrow \mathcal O_{k},\; T_{k-k'}(f)(z):=f(z-k'+k),\;\forall f\in\mathcal O_{k'},\Ree z>k',
\end{equation}
is a $\mathbb C$-algebra isomorphism. The above construction can be naturally extended as follows. For any $k<k'$, we have the natural maps of $\mathcal O_k$-algebras
$$ i_{k,k'}:\Omega(\NN,\mathcal O_k)\rightarrow \Omega(\NN,\mathcal O_{k'}),\; i_{k,k'}(\alpha)(n):=\alpha(n)|_{\Ree z> k'},\;\forall n\geq 1.$$
If we see $i_{k,k'}$'s as inclusions, we can define $\Omega^f(\NN,\mathcal O_{\infty}):=\bigcup_{k\in\mathbb R} \Omega(\NN,\mathcal O_k)$. 
Since $\alpha(n)\in \mathcal O_k$, for all $\alpha\in \Omega(\NN,\mathcal O_k)$, $k\in\mathbb R$ and $n\in \NN$, 
it follows that $\Omega(\NN,\mathcal O)=\bigcap_{k\in\mathbb R}\Omega(\NN,\mathcal O_k)$.
Note that
$$ \Omega(\NN,\mathcal O) \subset \Omega(\NN,\mathcal O_k) \subset \Omega^f(\NN,\mathcal O_{\infty})\subset \Omega(\NN,\mathcal O_{\infty}),
\forall k\in\mathbb R, $$
all the inclusion being strict. For any $k\in\mathbb R$, we consider {\small
$$\mathcal C_k:=\{\alpha\in \Omega(\NN, \mathcal O_k) :\forall\varepsilon>0, \exists M_{\varepsilon}:\{\Ree z>k\}\rightarrow [0,+\infty) \text{ continuous }$$ 
\begin{equation}\label{ceka}
\text{ such that } |\alpha(n)(z)|<n^{k+\varepsilon}M_{\varepsilon}(z),\forall n\geq 1,\Ree z>k \}. 
\end{equation}
As before, we can define $\mathcal C_{\infty}:=\bigcup_{k\in\mathbb R}\mathcal C_k$ and $\mathcal C:=\bigcap_{k\in\mathbb R}\mathcal C_k$. Note that
$$\mathcal C\subset \mathcal C_k \subset \mathcal C_{\infty} \subset \Omega^f(\NN,\mathcal O_{\infty}),\;\forall k\in\mathbb R,$$
where the inclusions are strict.
% of the domain of sequences of functions $\Omega(\NN,\mathcal O_k)$.

\begin{prop}
With the above notations, $\mathcal C_k$ is an $\mathcal O_k$-subalgebra of the domain $\Omega(\NN,\mathcal O_k)$.
\end{prop}

\begin{proof}
 Let $\alpha,\beta\in \mathcal C_k$ and let $\varepsilon>0$. Let $M_{\varepsilon},M'_{\varepsilon}:\{\Ree z>k\} \rightarrow [0,+\infty)$ such that
      $$|\alpha(n)(z)|\leq n^{k+\varepsilon}M_{\varepsilon}(z),\; |\beta(n)(z)|\leq n^{k+\varepsilon}M'_{\varepsilon}(z),\; \forall n\geq 1,\Ree z>k.$$
      It follows that $$|(\alpha+\beta)(n)(z)| = |\alpha(n)(z)|+|\beta(n)(z)|\leq n^{k+\varepsilon}(M_{\varepsilon}(z)+M'_{\varepsilon}(z)),\; \forall n\geq 1,\Ree z>k,$$
      hence $\alpha+\beta \in\mathcal C_k$. On the other hand, for $n\geq 1$ and $\Ree z>k$, we have that
      \begin{equation}\label{pula}|(\alpha\cdot \beta)(n)(z)| %= |\sum_{ab=n}\alpha(a)(z)\beta(b)(z)| 
        \leq \sum_{ab=n} |\alpha(a)(z)||\beta(b)(z)| \leq 
        \sum_{ab=n} a^{k+\varepsilon}M_{\varepsilon}(z)b^{k+\varepsilon}M'_{\varepsilon}(z) = d(n)n^{k+\varepsilon}M_{\varepsilon}(z)M'_{\varepsilon}(z), 
      \end{equation}
      where $d(n)$ is the number of (positive) divisors of $n$. For any $\varepsilon'>\varepsilon$, there exists a constant $C>0$ such that
      $d(n)< C n^{\varepsilon'-\varepsilon}$, for all $n\geq 1$, see \cite[Page 296]{apostol}. Let $\overline{M}_{\varepsilon'}(z)=C \cdot M_{\varepsilon}(z)M'_{\varepsilon}(z)$. By \eqref{pula} it follows that
      $$|(\alpha\cdot \beta)(n)(z)|\leq \overline{M}_{\varepsilon'}(z)n^{k+\varepsilon'},\;\forall n\geq 1,\Ree z>k,$$
      hence $\alpha\cdot\beta\in \mathcal C_k$. If $f\in\mathcal O_k$ and $\alpha\in \mathcal C_k$, then for any $\varepsilon>0$, we have that
      $$|(f\cdot \alpha)(n)(z)| = |f(z)||\alpha(n)(z)| \leq n^{k+\varepsilon}M_{\varepsilon}(z)|f(z)|,\forall n\geq 1,\Ree z>k,$$
      hence $f\cdot \alpha \in \mathcal C_k$.
\end{proof}

 % Also, we have the inclusions $ \mathcal C \subset \mathcal C_k \subset \mathcal C_{\infty}$.

 \begin{cor}
  With the above notations, it hold that
  \begin{enumerate}
   \item[(1)] $\mathcal C_{\infty}$ is an $\mathcal O_{\infty}$-subalgebra of $\Omega^f(\NN,\mathcal O_{\infty})$.
   \item[(2)] $\mathcal C$ is an $\mathcal O$-subalgebra of $\Omega(\NN,\mathcal O)$.
   \item[(3)] The inclusions $\Omega(\NN, \mathcal O) \subset \Omega(\NN, \mathcal O_k) \subset \Omega^f(\NN,\mathcal O_{\infty})$ induce
              the inclusions $\mathcal C \subset \mathcal C_k\subset \mathcal C_{\infty}$, i.e. $\mathcal C_k=\mathcal C_{\infty}\cap \Omega(\NN, \mathcal O_k)$ and 
               $\mathcal C=\mathcal C_{\infty} \cap \Omega(\NN, \mathcal O) = \mathcal C_{k} \cap \Omega(\NN, \mathcal O)$.
\end{enumerate}
\end{cor}

\begin{proof}
 It follows from Proposition $1.1$ and the above considerations by easy checkings.  
\end{proof}

Let $k\in \mathbb R$. We define
$$\mathcal E_k:=\{\LL\in \Omega(\NN,\mathcal O_{k})\;:\; \exists c>0,\;C:\{\Ree z>k\}\rightarrow [0,+\infty) \text{ continuous } $$
\begin{equation}\label{ek}
 \text{ such that } |\LL(n)(z)| \leq C(z)n^{- c(\Ree z)}, \;\forall n\geq 1,\Ree z>k\}.
\end{equation}
We let $\mathcal E_{\infty}:=\bigcup_{k\in\mathbb R}\mathcal E_k$ and $\mathcal E:=\bigcap_{k\in\mathbb R}\mathcal E_k$, where the intersection and the union
are naturally defined.

\begin{prop}
For any $k\in\mathbb R$, $\mathcal E_k$ has a structure of a $\mathbb C$-vector space, hence $\mathcal E_{\infty}$ 
and $\mathcal E$ have also structures of $\mathbb C$-vector spaces.
\end{prop}

\begin{proof}
 The proof is straightforward.
\end{proof}

\begin{prop}
Let $k\in \mathbb R$ and let $\LL\in \mathcal E_k$. There exists a constant $k'=k'(\LL)\geq k$ such that, 
for any $\alpha\in \mathcal C_k$, the series of functions 
$$F_{\LL}(\alpha)(z) := \sum_{n=1}^{+\infty}\alpha(n)(z)\LL(n)(z)$$
is uniformly absolutely convergent on the compact subsets of $\Ree z>k'$, hence $F(\alpha)\in \mathcal O_{k'}$.
\end{prop}

\begin{proof}
Let $\varepsilon>0$. From $(1.2)$ and \eqref{ek}, there exist a constant $c>0$ and two continuous maps $M_{\varepsilon},C:\{\Ree z>k\}\rightarrow [0,+\infty)$ such that
\begin{equation}\label{cucu1}
|\alpha(n)(z)\LL(n)(z)|\leq M_{\varepsilon}(z)n^{k+\varepsilon}C(z)n^{-c(\Ree z)} = M_{\varepsilon}(z)C(z)n^{k+\varepsilon-c\Ree z}
,\;\forall ;n\geq 1,\Ree z>k.
\end{equation}
Let $k':=\max\{\frac{k+1}{c},k\}$ and let $K\subset \{\Ree z>k'\}$ be a compact subset. 
Let $r:=\inf \{\Ree z\;:\;z\in K\}$. Since $K$ is compact and $\varepsilon>0$ can be arbitrarily chosen, we can assume that 
$\delta:=c(r-k')>\varepsilon$. From \eqref{cucu1} it follows that 
\begin{equation}\label{cucu2}
|\alpha(n)(z)\LL(n)(z)|\leq  M_{\varepsilon}(z)C(z) n^{-1-\delta+\varepsilon},\;\forall n\geq 1,\; z\in K.
\end{equation}
Let $M_K:=\sup\{M_{\varepsilon}(z)C(z):\; z\in K \}$. From \eqref{cucu2} it follows that
$$|\alpha(n)(z)\LL(n)(z)|\leq \frac{M_K}{n^{1+\delta-\varepsilon}},\;\forall n\geq 1,\;z\in K,$$
hence $F_{\LL}(\alpha)$ is uniformly absolutely convergent on $K$. Consequently, $F_{\LL}(\alpha)\in \mathcal O_{k'}$.
\end{proof}

\begin{teor}
Let $\LL\in \mathcal E_{\infty}$. 
\begin{enumerate}
\item[(1)] The map $F_{\LL}:\mathcal C_{\infty} \rightarrow \mathcal O_{\infty}$, $\alpha\mapsto F_{\LL}(\alpha)$, 
is a linear map of $\mathcal O_{\infty}$-modules. Moreover, if $\LL(1)(z)\neq 0$ for all $\Ree z\gg 0$, then $F_{\LL}$
is surjective.
\item[(2)] If $\LL(ab)=\LL(a)\LL(b)$, for all $a,b\in\NN$, and $\LL\neq 0$, then $F_{\LL}$
is a surjective morphism of $\mathcal O_{\infty}$-algebras with $F_{\LL}(f\cdot e)=f$, for all $f\in\mathcal O_{\infty}$.
\end{enumerate}
\end{teor}

\begin{proof}
$(1)$ Let $\alpha,\beta \in \mathcal C_{\infty}$, $f \in \mathcal O_{\infty}$ and $z\in \mathbb C$ with $\Ree z\gg 0$. 
      We have that 
      $$ F_{\LL}(\alpha+\beta)(z) = \sum_{n=1}^{+\infty}(\alpha+\beta)(n)(z)\LL(n)(z) = $$
      $$ = \sum_{n=1}^{+\infty}\alpha(n)(z)\LL(n)(z) + \sum_{n=1}^{+\infty}\beta(n)(z)\LL(n)(z) = 
         F_{\LL}(\alpha)(z)+F(\beta)_{\LL}(z) \text{ and } $$
      $$ F_{\LL}(f\cdot \alpha)(z) = \sum_{n=1}^{+\infty}(f\cdot \alpha)(n)(z)\LL(n)(z) = 
         f(z)\sum_{n=1}^{+\infty}\alpha(n)(z)\LL(n)(z) = f(z)F(\alpha)(z),$$
       thus $F$ is a linear map of $\mathcal O_{\infty}$-algebras.			
Now, assume $F_{\LL}(1)(z)\neq 0$, for all $\Ree z\gg 0$ and let 
$$g(z):=f(z)\LL(1)(z)^{-1},\; \Ree z\gg 0.$$
It follows that $F_{\LL}(g\cdot e)=f$, hence $F_{\LL}$ is surjective.

$(2)$ Let $\alpha,\beta \in \mathcal C_{\infty}$ and $z\in \mathbb C$ with $\Ree z\gg 0$. We have that
 $$ F_{\LL}(\alpha\cdot \beta)(z)= \sum_{n=1}^{+\infty}(\alpha\cdot \beta)(n)(z)\LL(n)(z) = 
     \sum_{n=1}^{+\infty}\sum_{ab=n}\alpha(a)(z)\beta(b)(z)\LL(ab)(z) = $$
  $$ = \sum_{a=1}^{+\infty}\alpha(a)(z)L(a)(z)\sum_{b=1}^{+\infty}\beta(b)(z)L(b)(z)=F(\alpha)(z)F(\beta)(z),$$
 therefore  $F_{\LL}$ is a morphism of $\mathcal O_{\infty}$-algebras. Moreover, the hypothesis implies 
 $\LL(1)(z)=\LL(1)(z)^2$ so either $\LL(1)(z)=0$, either $\LL(1)(z)=1$. If the zero set of $\LL(1)$ is nondiscrete,
 then $\LL(1)$ is identically zero, which implies $\LL(n)=\LL(1)\LL(n)$ is identically zero, for all $n\geq 1$, a
 contradiction. Therefore, $\LL(1)$ takes the value $1$ in a nondiscrete subset of the half-plane $\Ree z\gg 0$, hence 
 $\LL(1)(z)=1$, for all $z\in \mathbb C$.		
\end{proof}

Let $k\in \mathbb R$. We consider the following subset
 $$
 \widetilde{\mathcal E}_k:=\{\LL\in \mathcal E_k\;:\; \LL(n)(z)\neq 0,\;\forall n\geq 1,\;\Ree z>k \text{ and } 
$$
\begin{equation}\label{ekt} 
 \forall n_0\geq 1,\; \exists  C(n_0)>0, \text{ such that } \frac{|L(n)(z)|}{|L(n_0)(z)|} \leq n^{-C(n_0)\Ree z},\;\forall n\geq n_0+1,\; \Ree z>k \}.
\end{equation}
 We consider also $\widetilde{\mathcal E}_{\infty}:=\bigcup_{k\in\mathbb R}\widetilde{\mathcal E}_k 
\text{ and } \widetilde{\mathcal E}:=\bigcap_{k\in\mathbb R}\widetilde{\mathcal E}_k$, where the intersection and the union
are naturally defined.

\begin{obs} \emph{
Let $\lambda:\NN \rightarrow (0,+\infty)$ be an increasing sequence of positive real numbers such that 
\begin{equation}\label{lambda}
 \limsup_{n\rightarrow+\infty}\frac{\log n}{\lambda (n)} < + \infty.
\end{equation}
We define 
$\LL:=e^{-\lambda}:\NN \rightarrow \mathcal O,\; \LL(n)(z):=e^{-\lambda(n)z}, \;z\in\mathbb C.$
The condition \eqref{lambda} implies that there exists a constant $c>0$ such that
\begin{equation}\label{lambda2}
 \lambda(n)\geq c\log n,\;\forall n\geq 0.
\end{equation}
From \eqref{lambda2} it follows that 
%\begin{equation}
 $$ |\LL(n)(z)|\leq e^{-c(\Ree z)\log n} = n^{-c\Ree z},\;\forall z\in\mathbb C,n\geq 1,$$
%\end{equation}
hence $\LL\in \mathcal E$. Also $L(n)(z)\neq 0$, for all $z\in \mathbb C$. Let $n_0\in\NN$ and $n\geq n_0+1$.
It holds that
\begin{equation}
\left|\frac{\LL(n)(z)}{\LL(n_0)(z)}\right| = e^{-(\lambda(n)-\lambda(n_0))\Ree z},\;\forall z\in\mathbb C.
\end{equation}
By $(1.9)$, there exists a constant $C(n_0)>0$ such that 
$$ e^{-(\lambda(n)-\lambda(n_0))\Ree z} \leq e^{-C(n_0)\log n \Ree z} = n^{-C(n_0)\Ree z},\;\forall z\in \mathbb C,\;\Ree z>0,$$
hence, from $(1.10)$, it follows that $\LL\in \widetilde{\mathcal E_0}$.}
% Choosing $\delta(n)=\lambda_{n+1}-\lambda_n$, $n\geq 1$, by $(1.12)$, it follows that $\LL \in \widetilde{\mathcal E}_k$.
\emph{We let
$$O(n^k)=\{ \alpha\in\Omega(\NN,\mathbb C)\;|\;\exists C>0, \text{ such that} |\alpha(n)|\leq Cn^{k},\;\forall n\geq 1\},$$
the set of arithmetical functions of order $n^k$. We consider the $\mathbb C$-algebras 
\begin{equation}
\mathcal D :=\mathcal C\cap \Omega(\NN,\mathbb C), \mathcal D_k := \mathcal C_k\cap \Omega(\NN,\mathbb C) \text{ and } \mathcal D_{\infty} := \mathcal C_{\infty} \cap \Omega(\NN,\mathbb C).
\end{equation}
It is easy to check that 
$$\mathcal D_k=\bigcap_{\varepsilon>0} O(n^{k+\varepsilon}), \; \mathcal D_{\infty} = \bigcup_{k\in \mathbb R} O(n^k) \text{ and } \mathcal D=\bigcap_{k\in\mathbb R} O(n^k).$$
If $\alpha\in \mathcal D_k$ then the \emph{general Dirichlet series}
$$
F_{\lambda}(\alpha)(z):=F_{\LL}(\alpha)(z)=\sum_{n=1}^{+\infty}\alpha(n)e^{-\lambda(n)z},
$$
defines a holomorphic function on $\Ree z>k'$.}

\emph{
It is well known \cite[Theorem 6]{hardy} that $F_{\lambda}(\alpha)$ is identically zero 
if and only if $\alpha(n)=0$, for all $n\geq 1$. Consequently, from Theorem $1.5(1)$, the map
$F_{\lambda}:\mathcal D_{\infty}\rightarrow \mathcal O_{\infty},\; \alpha\mapsto F_{\lambda}(\alpha)$,
is an injective morphism of $\mathbb C$-vector spaces.
In particular, for $\lambda(n)=\log n$, the \emph{(classical) Dirichlet series}
$$
F(\alpha)(z):=\sum_{n=1}^{+\infty}\frac{\alpha(n)}{n^z},
$$
defines a holomorphic function on $\Ree z>k+1$. From the above remarks or as a consequence of the uniqueness theorem for Dirichlet series (\cite[Theorem 11.3]{apostol}),
$F(\alpha)$ is identically zero if and only if $\alpha(n)=0$, for all $n\geq 1$. Hence, from Theorem $1.5(2)$, the map
$F_{\lambda}:\mathcal D_{\infty}\rightarrow \mathcal O_{\infty},\; \alpha\mapsto F_{\lambda}(\alpha),$
is an injective morphism of $\mathbb C$-algebras.}
\end{obs}

\section{Main results}

Let $k\in\mathbb R$. Similarly to \cite[page 2]{cim}, we define
\begin{equation}
\Be := \{f \in \Ok \;:\;\forall a>0,\; \lim_{x\rightarrow +\infty}e^{-ax}|f(x)|=0 \}, 
\end{equation}
\begin{equation}
\Ie := \{f\in \Ok \;:\exists a>0
\text{ such that} \;\lim_{x \rightarrow +\infty}e^{ax}|f(x)|=0 \}.
\end{equation}
We also let
\begin{equation}
\mathcal B_{\infty}:=\bigcup_{k\in\mathbb R}\Be \subset \mathcal O_{\infty},\; \mathcal B:=\bigcap_{k\in\mathbb R}\Be\subset \mathcal O,\; \mathcal I_{\infty}:=\bigcup_{k\in\mathbb R}\Ie \subset \mathcal O_{\infty} 
\text{ and }\;\mathcal I :=\bigcap_{k\in\mathbb R}\Ie \subset \mathcal O.
\end{equation}

\begin{prop}
With the above notations, it hold that
\begin{enumerate}
\item[(1)] $\Be$ is a subdomain of $\Ok$ and $\Ie$ is an ideal of $\Be$.
\item[(2)] $\mathbb C[z]\subset \Be$ and $\Ie \cap \mathbb C[z] = \{0\}$. In particular, $\Be$ is a $\mathbb C[z]$-subalgebra of $\Ok$.
\item[(3)] If $k,k'\in\mathbb R$, then $T_{k-k'}:\Be \rightarrow \Bep$, $T_{k-k'}(f)(z):=f(z+k-k')$, for all $f\in\Bep$, is a $\mathbb C$-algebra isomorphism.
\end{enumerate}
\end{prop}

\begin{proof}
$(1)$ % It is clear that the constant functions are in $\Be$. 
Let $f,g\in \Be$ and $a>0$. We have
\begin{equation}
0\leq e^{-ax}|f(x)+g(x)|\leq e^{-ax}|f(x)| + e^{-ax}|g(x)|,\;\forall x>k. 
\end{equation}
By $(2.1)$ and $(2.4)$ it follows that
$$ \lim_{x\rightarrow +\infty}e^{-ax}|f(x)+g(x)|=0, $$
hence $f+g\in \Be$. On the other hand, by $(2.1)$, we have
$$\lim_{x\rightarrow +\infty}e^{-ax}|f(x)g(x)| = \lim_{x\rightarrow +\infty}e^{-\frac{a}{2}x}|f(x)|
\lim_{x\rightarrow +\infty}e^{-\frac{a}{2}x}|g(x)| = 0,$$
hence $f\cdot g \in \Be$. Therefore $\Be$ is a subdomain of $\Ok$.

Let $f,g\in \mathcal I_{k}$. From $(2.2)$, there exists $a>0$ such that
$$\lim_{x\rightarrow +\infty}e^{ax}|f(x)| = \lim_{x\rightarrow +\infty}e^{ax}|g(x)| = 0.$$
Since $e^{ax}|f(x)+g(x)|\leq e^{ax}|f(x)|+e^{ax}|g(x)|$, it follows that
$$\lim_{x\rightarrow +\infty}e^{ax}|f(x)+g(x)|=0,$$
hence $f+g\in\Ie$. Let $f\in \mathcal I_{k}$ and $g\in \mathcal B_{k}$. Let $a>0$ such that 
\begin{equation}
\lim_{x\rightarrow +\infty}e^{ax}|f(x)|=0.
\end{equation}
By $(2.4)$ and $(2.1)$ it follows that
$$\lim_{x \rightarrow +\infty} e^{\frac{a}{2}x}|f(x)g(x)| =  \lim_{x \rightarrow +\infty} e^{ax}|f(x)| e^{-\frac{a}{2}x} |g(x)| = 0,$$
hence $f\cdot g\in \Ie$.

$(2)$ Let $f\in \mathbb C[z]$. It is clear that
$$ \lim_{x\rightarrow +\infty}e^{-ax}|f(x)| = 0,\;\forall a>0, $$ 
hence $f\in\Be$. On the other hand, if $f\neq 0$, then $\lim_{x\rightarrow +\infty}|f(x)|>0$, therefore $f\notin\Ie$.

$(3)$ It follows by straighforward computations, similarly to formula \eqref{tkk}.
\end{proof}

\begin{cor}
 With the above notations, it hold that
\begin{enumerate}
 \item[(1)] $\mathcal B_{\infty}$ is a $\mathbb C[z]$-subalgebra of $\mathcal O_{\infty}$ and $\mathcal I_{\infty}$ is an ideal in $\mathcal B_{\infty}$ with $\mathcal I_{\infty}\cap \mathbb C[z]=\{0\}$.
 \item[(2)] $\mathcal B$ is a $\mathbb C[z]$-subalgebra of $\mathcal O$ and $\mathcal I$ is an ideal in $\mathcal B$ with $\mathcal I\cap \mathbb C[z]=\{0\}$.
\end{enumerate}
\end{cor}

\begin{proof}
 It follows immediately from $(2.3)$ and Proposition $2.1$.
\end{proof}

\begin{prop}
For any $a>0$, let $f_a(z):=e^{-az}$, $z\in\mathbb C$. It hold that
\begin{enumerate}
\item[(1)] If $f\in \Be$ then $f\in \Ie \Leftrightarrow$ there exists $a>0$ such that $g=\frac{f}{f_a}\in \Ie$.
\item[(2)] If $a<b$ then $f_a\Ie \supsetneq f_b\Be$.
\item[(3)] $\Ie = \sum_{a>0}f_a\Be = \bigcup_{a>0}f_a\Be$.
\item[(4)] Let $(a_n)_{n\geq 1}$ be a sequence of positive real numbers with $\liminf_n a_n=0$. Then
           $\{f_{a_n}:n\geq 1\}$ is a system of generators of the ideal $\Ie$.
\item[(5)] The ideal $\Ie$ is not finitely generated.
\end{enumerate}
\end{prop}

\begin{proof}
$(1)$ First, note that $f_a \in \Ie$, for any $a>0$. Let $f\in \Ie$. Then there exists $a>0$ such that
\begin{equation}
\lim_{x\rightarrow +\infty}e^{2ax}|f(x)| = 0.
\end{equation}
 Let $g=\frac{f}{f_a}$, that is $g(z)=e^{az}f(z)$. From $(2.6)$ it follows that 
$$ \lim_{x\rightarrow +\infty}e^{ax}|g(x)| = 0, $$
thus $g \in\Ie$. The other implication is obvious, since $\Ie$ is an ideal.

$(2)$ Let $f\in f_b\Be$. It follows that there exists $g\in\Be$ such that $f(z)=e^{-bz}g(z)$.
      Let $h(z)=e^{(a-b)z}g(z)$. We have $h\in \Ie$ and $f(z)=f_a(z)h(z)$. Thus $f\in f_a\Ie$.
			On the other hand, $f_{\frac{a+b}{2}}\in f_a\Ie \setminus f_b\Be$.

$(3)$ The set $\{ f_a\Be\;:\;a>0\}$ is totally ordered by inclusion.

$(4)$ Is a direct consequence of $(2)$ and $(3)$.

$(5)$ Assume that $\{f_1,\ldots,f_m\}$ is a minimal set of generators of $\Ie$. Then there exists $a>0$ such
      that $f_i(z)=f_a(z)g_i(z)$ with $g_i\in \Ie$, for $1\leq i\leq m$. It follows that $(f_1,\ldots,f_m)\Be \subset f_a\Be \subsetneq \Ie$,
			a contradiction.
\end{proof}

\begin{obs}\emph{
$(1)$ The ideal $\Ie$ is not prime. 
Let $f(z):=\sin(\sin z + 1 + e^{-z})$, $z\in\mathbb C$. It is easy to see that
$$ \limsup_{x\rightarrow +\infty} |f(x)| = 1\;\text{and } \liminf_{x\rightarrow +\infty} |f(x)| = 0,$$
hence $f\in \mathcal B_k \setminus \mathcal I_k$. Let $g:\mathbb R\rightarrow (0,+\infty)$,
$$ g(x)=\begin{cases} \frac{e^{-x}}{f(x)},& x\geq 0 \\ \sin (2),& x<0 \end{cases}.$$
The function $g$ is continuous on $\mathbb R$. According to a theorem of Carleman \cite{carleman}, 
there exists an entire function $h:\mathbb C\rightarrow \mathbb C$ such that
\begin{equation}\label{carle1}
|h(x)-g(x)| < e^{-x},\;\forall x\in\mathbb R. 
\end{equation}
We prove that $h\in \mathcal B_k$ and $fh\in \mathcal I_k$. By straightforward computations, we get
\begin{equation}\label{carle2}
 \limsup_{x\rightarrow +\infty} g(x) = 1\;\text{and } \liminf_{x\rightarrow +\infty} g(x) = 0.
\end{equation}
By \eqref{carle1} and \eqref{carle2} it follows that 
$$ \limsup_{x\rightarrow +\infty} h(x) = 1\;\text{and } \liminf_{x\rightarrow +\infty} h(x) = 0,$$
hence $h\in \Be \setminus \Ie$. On the other hand, from \eqref{carle1} it follows that
$$ |f(x)h(x)| < |f(x)g(x)| + |e^{-x}f(x)| < e^{-x}|1+f(x)| \leq 2e^{-x},\;\forall x\in\mathbb R,$$
hence $fh\in \mathcal I_k$ as required.}

\emph{$(2)$ If $f\in\Be$ is inversible, then $f(z)\neq 0$, for all $\Ree z>k$, and $f\notin \Ie$. The first condition must be satisfied in order that 
$\frac{1}{f}$ to be defined on $\{\Ree z>k\}$. The second condition follows from the fact that $\Ie$ is a proper ideal of $\Be$.
On the other hand, if $f(z):=e^{-(\sin z+1)z}$, $z\in\mathbb C$, then $f\in \Be\setminus \Ie$, $f(z)\neq 0$, for all $z\in\mathbb C$, but $\frac{1}{f}\notin \Be$.}

\emph{$(3)$ The results of Proposition $2.3$ and the previous remarks are valid for $\mathcal B_{\infty}, \mathcal I_{\infty}, \mathcal B$ and $\mathcal I$.}
\end{obs}

Let $f\in \mathcal O$ be an entire function. If there exist a positive number $\rho$ and
constants $A,B > 0$ such that
\begin{equation}\label{ord}
|f(z)|\leq Ae^{B|z|^{\rho}} \text{ for all }z\in\mathbb C, 
\end{equation}
then we say that f has an \emph{order of growth} $\leq\rho$. We define the \emph{order of growth} of $f$ as
$$\rho(f)=\inf\{\rho>0\;:\;f \text{ has an order of growth }\leq\rho \}.$$
For any $\rho>0$, we let $\mathcal O_{<\rho}$ the set of entire functions or order of growth $<\rho$.
It is easy to check that $\mathcal O_{<\rho}$ is a $\mathbb C$-subalgebra of the algebra of entire functions $\mathcal O$. 
The following result was proved, in a different context, in \cite[Proposition 5]{cim}. In order of 
completeness, we present here the proof.

\begin{prop}
With the above notations, $\mathcal O_{<1}$ is a $\mathbb C$-subalgebra of $\mathcal B$.
Moreover, $\mathcal O_{<1}\cap \mathcal I = \{0\}$
\end{prop}

\begin{proof}
Let $f\in\mathcal O_{<1}$ and assume that $f$ has an order of growth $\leq\rho<1$.
Let $a>0$ and $x>0$. According to \eqref{ord}, there exist two constants $A,B>0$ such that
$$ e^{-ax}|f(x)|\leq Ae^{-ax+Bx^{\rho}},\;\text{ so } \lim_{x\rightarrow+\infty}e^{-ax}|f(x)|=0,$$
hence $f\in\mathcal B$. If $f$ is polynomial then, by Corollary $2.2(2)$, $f\in \mathcal I \Leftrightarrow f=0$.
Assume $f$ is not polynomial. Then, by Hadamard's Theorem, there exists $D\in\mathbb C^*$ such that
\begin{equation}\label{cutu1}
f(z)=Dz^mE(z),\;\text{ with }E(z)=\prod_{n=1}^{+\infty}(1-\frac{z}{z_n}),z\in\mathbb C,
\end{equation}
where $m$ is the multiplicity of $z_0=0$ as zero of $f$ and $z_1,z_2,\ldots$ are the non-zero zeros of $f$.
According to \cite[Corollary 5.4]{stein}, there exists a strictly increasing sequence $(x_k)_{k\geq 1}$ of positive numbers with
$\lim_{k\rightarrow+\infty}x_k=+\infty$ and a constant $B'>0$ such that
\begin{equation}\label{cutu2}
|E(x_k)|\geq e^{-B'x_k^{\rho}},\;\forall k\geq 1.
\end{equation}
Let $a>0$. From \eqref{cutu1} and \eqref{cutu2} it follows that
$$e^{ax_k}|f(x_k)| = e^{ax_k}|D|x_k^m|E(x_k)| \geq |D|x_k^me^{ax_k-B'x_k^{\rho}} \rightarrow +\infty,$$
hence $f\notin \mathcal I$.
\end{proof}

\noindent
The following lemma, which generalize \cite[Lemma 1]{cim}, is a key part in the proof of Theorem $2.7$.

\begin{lema}
Let $\alpha\in\Omega^f(\NN,\mathcal O_{\infty})=\bigcup_{k\in\mathbb R}\Omega(\NN,\mathcal O_k)$ with $\alpha(n)\notin \mathcal I_{\infty}\setminus\{0\}$, for all $n\geq 1$, 
such that there exists $k\in \mathbb R$ and a continuous function $M:\{\Ree z> k\}\rightarrow [0,+\infty)$
which satisfies 
\begin{enumerate}
 \item[(i)] $|\alpha(n)(z)|\leq M(z)n^{k}$, for all $n\geq 1,\Ree z>k$.
\item[(ii)] $\lim_{x\rightarrow+\infty} e^{-ax}M(x) = 0$, for all $a>0$.
\end{enumerate}
The following hold
\begin{enumerate}
 \item[(1)] $\alpha\in\mathcal C_k \cap \Omega(\NN,\mathcal B_k)$.
 \item[(2)] Let $L \in\widetilde{\mathcal E}_{\infty}$. If $\alpha(n)\notin \Ie\setminus\{0\}$, for all $n\geq 1$,
       and $F_L(\alpha)(z):=\sum_{n=1}^{+\infty}\alpha(n)(z)L(n)(z)$ is identically zero, then $\alpha=0$.
\end{enumerate}
\end{lema}

\begin{proof}
$(1)$ The hypothesis $(i)$ implies $\alpha\in \mathcal C_k$. Also, $(i)$ and $(ii)$ implies $\alpha(n) \in\mathcal B_k$, for all $n\geq 1$.

$(2)$ Note that, according to Proposition $1.4$, $F_L(\alpha)$ is defined on $\Ree z\gg 0$.
      Let $n_0$ be the smallest integer with $\alpha(n_0)\neq 0$. Since $F_L(\alpha)=0$, for any $x\gg 0$ we have that
\begin{equation}\label{coor}
 |\alpha(n_0)(x)| = \left| \sum_{n=n_0+1}^{+\infty}\alpha(n)(x)\frac{L(n)(x)}{L(n_0)(x)} \right| \leq \sum_{n=n_0+1}^{+\infty}|\alpha(n)(x)|\left|\frac{L(n)(x)}{L(n_0)(x)} \right|.
\end{equation}
Since $L \in\widetilde{\mathcal E}_{\infty}$, by \eqref{ekt}, there exists $C(n_0)>0$ such that 
\begin{equation}\label{pulas}
 \left|\frac{L(n)(x)}{L(n_0)(x)}\right| \leq n^{-C(n_0)x}, \forall x \gg 0,\;n\geq n_0+1 
\end{equation}
From \eqref{coor} and \eqref{pulas} it follows that
\begin{equation}\label{pulaas}
 |\alpha(n_0)(x)| \leq M(x) \sum_{n=n_0+1}^{+\infty} e^{(-C(n_0)x+k)\log n}, \forall x\gg 0.
\end{equation}
Let $0<2a<C(n_0)$. From \eqref{pulaas} it follows that
\begin{equation}\label{pulaaa}
 e^{ax}|\alpha(n_0)(x)| \leq e^{-ax}M(x) \sum_{n=n_0+1}^{+\infty} e^{((2a-C(n_0))x+k)\log n}, \forall x\gg 0.
\end{equation}
Taking $\lim_{x\rightarrow +\infty}$ in \eqref{pulaaa}, by hypothesis $(ii)$, it follows that 
$$\lim_{x\rightarrow +\infty}e^{ax}|\alpha(n_0)(x)| = 0,$$
hence $\alpha(n_0)\in \mathcal I_{\infty}$. Therefore $\alpha(n_0)=0$.
\end{proof}

Let $\mathcal A_{\infty} \subset (\mathcal B_{\infty}\setminus \mathcal I_{\infty})\cup \{0\}$ be a 
$\mathbb C$-subalgebra of $\mathcal B_{\infty}$. (According to Proposition $2.5$ we can choose $\mathcal A_{\infty}$ to be the domain of entire functions of order $<1$). Let $\mathcal A_k:=\mathcal A_{\infty}\cap \mathcal O_k$, 
$k\in\mathbb R$, and $\mathcal A:=\mathcal A_{\infty}\cap \mathcal O$.
In the $\mathcal A_{k}$-algebra $\Omega(\NN,\mathcal A_{k})$ we consider the $\mathcal A_{k}$-subalgebra defined by
$$
\mathcal H_{k}:=\{ \alpha\in \Omega(\NN,\mathcal A_{k})\;:\;\exists M:\{\Ree z> k\}\rightarrow [0,+\infty) \text{ continuous } 
$$
\begin{equation}
\text{ such that (i)} |\alpha(n)(z)|\leq M(z)n^{k}, \forall n\geq 1,\Ree z>k,\text{ (ii) }\lim_{x\rightarrow+\infty} e^{-ax}M(x) = 0 \forall a>0\}.
\end{equation}
Let $\mathcal H_{\infty}:=\bigcup_{k\in\mathbb R}\mathcal H_k$ and $\mathcal H:=\bigcap_{k\in\mathbb R}\mathcal H_k$.
Note that, from $(1.11)$ and $(2.16)$ it follows that $\mathcal D_k\subset \mathcal H_k$, for all $k\in\mathbb R$, hence $\mathcal D_{\infty}\subset \mathcal H_{\infty} \text{ and } 
\mathcal D\subset \mathcal H$. 
% The main result of this section is the following theorem, which generalize Proposition $1.6$.

\begin{teor}
Let $\LL \in \widetilde{\mathcal E}_{\infty}$. It hold that
\begin{enumerate}
\item[(1)] The map $F_{L}:\mathcal H_{\infty} \rightarrow \mathcal O_{\infty}$ is an injective morphism of $\mathcal A_{\infty}$-modules.
\item[(2)] If $\LL(ab)=\LL(a)\LL(b)$, for all $a,b\in\NN$, and $\LL\neq 0$, then
the map $F_L$ is an injective morphism of $\mathcal A_{\infty}$-algebras.
\end{enumerate}
\end{teor}

\begin{proof}
It follows from Theorem $1.5$, $(2.16)$ and Lemma $2.6$.
\end{proof}

Let $\mathcal F_{\infty}$ be the quotient field of $\mathcal A_{\infty}$. Let $\mathcal F_k$ bet the quotient field of $\mathcal A_k$
and $\mathcal F$ be the quotient field of $\mathcal A$. It is easy to check that $\mathcal F_{\infty}=\bigcup_{k\in\mathbb R}\mathcal F_k$ and $\mathcal F =\bigcap_{k\in\mathbb R}\mathcal F_k$.
Note that, if $\mathcal A_{\infty}=\mathcal A=\mathcal O_{<1}$ is the $\mathbb C$-algebra of entire functions of order $<1$, then $\mathcal F_{\infty}=\mathcal F$ is the field of meromorphic functions of order $<1$.

\begin{cor}
Let $\alpha_1,\ldots,\alpha_r \in \mathcal H_{\infty}$. Assume there exists a nondiscrete subset $S\subset \{\Ree z\gg 0\}$ such that
the numerical sequences $n\mapsto \alpha_j(n)(z)$, $1\leq j\leq r$, are linearly independent over $\mathbb C$, for any $z\in S$.
Let $ L\in\widetilde{\mathcal E}_{\infty}$. Then
\begin{enumerate}
\item[(1)] $F_L(\alpha_1),\ldots,F_L(\alpha_r)\in \mathcal O_{\infty}$ are linearly independent over $\mathcal F_{\infty}$.
\item[(2)] Moreover, if $L(ab)=L(a)L(b)$, for all $a,b\in\NN$, $L\neq 0$, and $n\mapsto \alpha_j(n)(z)$, $1\leq j\leq r$, are algebraically independent over $\mathbb C$, for any $z\in S$,
then $F_L(\alpha_1),\ldots,F_L(\alpha_r)\in \mathcal O_{\infty}$ are algebraically independent over $\mathcal F_{\infty}$.
\end{enumerate} 
\end{cor}

\begin{proof}
$(1)$ Let $g_1,\ldots,g_r \in \mathcal F_{\infty}$ such that
$g_1\alpha_1 + \cdots + g_r\alpha_r = 0$. Multiplying this with a common multiple of the denominators of $g_1,\ldots,g_r$, we
can assume that $g_1,\ldots,g_r\in \mathcal A_{\infty}$. It follows that 
\begin{equation}\label{kuku}
g_1(z)\alpha_1(n)(z)+\cdots+g_r(z)\alpha_r(n)(z) = 0,\;\forall \Ree z\gg 0, n\geq 1.
\end{equation}
The hypothesis and \eqref{kuku} implies $g_1(z)=\cdots=g_r(z)=0$, for all $z\in S$, hence, by the identity
theorem of holomorphic functions, it follows that $g_1=\cdots=g_r=0$, therefore $\alpha_1,\ldots,\alpha_r$
are linearly independent over $\mathcal A_{\infty}$. Now, Theorem $2.7(1)$ implies $(1)$.

$(2)$ Let $I\subset \mathbb N^r$ be a finite subset of indices and let $g_i\in\mathcal A_{\infty}$, $i\in I$.
Assume that $$\sum_{(i_1,\ldots,i_r)\in I}g_i(z) (\alpha_1^{i_1}\cdots \alpha_r^{i_r})(n)(z) = 0,\;\forall \Ree z\gg 0,n\geq 1.$$ 
The hypothesis implies $g_i(z)=0$, for all $z\in S$ and $i\in I$, hence the holomorphic functions $g_i$'s are identically zero.
The conclusion follows from Theorem $2.7(2)$.
\end{proof}

\begin{cor}
Let $\alpha_1,\ldots,\alpha_r\in \mathcal D_{\infty}$, linearly independent over $\mathbb C$, and 
let $L\in\widetilde{\mathcal E}_{\infty}$. Then $F_L(\alpha_1),\ldots$, $F_L(\alpha_r)\in \mathcal O_{\infty}$ are linearly independent over 
$\mathcal F_{\infty}$. Moreover, if $L(ab)=L(a)L(b)$, for all $a,b\in\NN$, $L\neq 0$, and 
$\alpha_1,\ldots,\alpha_r\in \mathcal D_{\infty}$ are algebraically independent over $\mathbb C$,
then $F_L(\alpha_1),\ldots$, $F_L(\alpha_r)\in \mathcal O_{\infty}$ are algebraically independent over 
$\mathcal F_{\infty}$.
\end{cor}

\begin{proof}
It follows from Corollary $2.8$ and the inclusion $\mathcal D_{\infty}\subset \mathcal H_{\infty}$.
\end{proof}

Note that, according to Remark $1.6$, the case of (general) Dirichlet series is contained in Corollary $2.9$.

\section{Applications to Dirichlet series associated to multiplicative functions}

Given an arithmetic function $\alpha\in\Omega(\NN,\mathbb C)$ and a non-negative integer $j$ its arithmetic $j$-derivative is
$$\alpha^{(j)}(n) := (-1)^j \alpha(n) \log^j n,\;\forall n\geq 1.$$
Assume $\alpha\in\mathcal D_k$. Since $\log n$ has the order of growth $O(n^{\varepsilon})$ for any $\varepsilon>0$,
it follows, by straighforward computations, that $\alpha^{{(j)}}\in\mathcal D_k$. Moreover, if $\alpha\in \mathcal D_k$, then
the $j$-derivative of the Dirichlet series $F(\alpha)=\sum_{n=1}^{+\infty} \frac{\alpha(n)}{n^z}$ is
\begin{equation}
F^{(j)}(\alpha)(z)=F(\alpha^{(j)})(z),\;\forall \Ree z>k+1.
\end{equation}
An arithmetic function $\alpha\in\Omega(\NN,\mathbb C)$ is called \emph{multiplicative}, if
$\alpha(1)=1$ and
$$ \alpha(nm)=\alpha(n)\alpha(m),\;\forall n,m\in\mathbb N\;\text{with }\gcd(n,m)=1.$$
Two multiplicative arithmetic functions $\alpha,\beta \in \Omega(\NN,\mathbb C)$ are 
\emph{equivalent}, see \cite{kac}, if 
$f(p^j)=g(p^j)$ for all integers $j\geq 1$ and all but finitely many
primes $p$. We recall that $e\in \Omega(\NN,\mathbb C)$, defined by $e(1)=1$ and $e(n)=0$ for $n\geq 2$, is the identity function. Obviously, $e$ is multiplicative.
We recall the following result of Kaczorowski, Molteni and Perelli \cite{kac}.

\begin{lema}(\cite[Lemma 1]{kac})
Let $\alpha_1,\ldots,\alpha_r \in \Omega(\NN,\mathbb C)$ be multiplicative functions such that
$e,\alpha_1,\ldots,\alpha_r$ are pairwise non-equivalent, and let
$m$ be a non-negative integer. Then the functions
$$\alpha_1^{(0)},\ldots,\alpha_1^{(m)}, \alpha_2^{(0)},\ldots,\alpha_2^{(m)},
\ldots, \alpha_r^{(0)},\ldots,\alpha_r^{(m)} \in \Omega(\NN,\mathbb C)$$
are linearly independent over $\mathbb C$.
\end{lema}

\begin{prop}(See also \cite[Corollary 4]{cim} and \cite[Corollary 6]{cim})
Let $\alpha_1,\ldots,\alpha_r \in \mathcal D_{k}$ be multiplicative functions 
such that $e,\alpha_1,\ldots,\alpha_r$ are pairwise non-equivalent. Let $m$ be a non-negative integer.
Then 
$$ F^{(0)}(\alpha_1),\ldots,F^{(m)}(\alpha_1),F^{(0)}(\alpha_2),\ldots,F^{(m)}(\alpha_2), \ldots, F^{(0)}(\alpha_r),\ldots,F^{(m)}(\alpha_r) $$
are linearly independent over $\mathcal F_{k+1}$, hence, in particular,
over the field of meromorphic functions of order $<1$.
% , where $D_j^{(k)}(z)$ is the $k$-th derivative of $D_j(z)$.
\end{prop}

\begin{proof}
If follows from $(3.1)$, Lemma $3.1$ and Corollary $2.9$.
\end{proof}

%\begin{cor}(See also \cite[Corollary 6]{cim})
%In the conditions of the Proposition $3.2$, the functions 
%$$ F^{(0)}(\alpha_1),\ldots,F^{(m)}(\alpha_1),F^{(0)}(\alpha_2),\ldots,F^{(m)}(\alpha_2), \ldots, F^{(0)}(\alpha_r),\ldots,F^{(m)}(\alpha_r) $$
%are linearly independent over the field of meromorphic functions of order $<1$.
%\end{cor}

%\begin{proof}
% We use Proposition $3.2$ and Proposition $2.5$.
%\end{proof}

Note that Proposition $3.2$, combined with \cite[Lemma 2]{kac}, generalize the main result in \cite{kac}.

\begin{prop}
If $\alpha_1,\ldots,\alpha_r\in \mathcal D_k$ are multiplicative functions, algebraically independent over $\mathbb C$, then 
$F(\alpha_1),\ldots,F(\alpha_r)\in \mathcal O_{k+1}$ are algebraically independent over $\mathcal F_{k+1}$, hence, in particular,
over the field of meromorphic functions of order $<1$.
\end{prop}

\begin{proof}
It is a special case of the second part of Corollary $2.9$.
\end{proof} 

\begin{obs}\emph{
Let $K/\mathbb Q$ be a finite Galois extension. Let $\chi_1,\ldots,\chi_h$ be the irreducible characters of the Galois group.
It was proved in \cite[Corollary 5]{florin}, that the L-Artin functions, see \cite{artin1}, $L(z,\chi_1),\ldots, L(z,\chi_h)$ associated to $\chi_1,\ldots,\chi_h$
are algebraic independent over $\mathbb C$. This result was extended in \cite[Corollary 9]{cim} for the field of meromorphic functions
of order $<1$. Assuming $L(z,\chi_j)=F(\alpha_j)(z)$, $1\leq j\leq h$, $\Ree z>1$, the key point in the proof of the above results was to show that $\alpha_1,\ldots,\alpha_h\in \mathcal D_{\varepsilon}$, 
where $\varepsilon>0$ can by arbitrarely choosed, are in fact algebraically independent over $\mathbb C$. Therefore, $\alpha_1,\ldots,\alpha_h$ satisfy the hypothesis of Proposition $3.3$ for $k=\varepsilon$.}
\end{obs}

\noindent
\textbf{Aknowledgment.} I express my gratitude to Florin Nicolae for valuable discussions regarding the results of this paper.

\section*{Compliance with Ethical Standards.} 

This article was not funded. The author declares that he has not conflict of interest. 
% This article does not contain any studies with human participants or animals performed by any of the authors.

{}

\vspace{2mm} \noindent {\footnotesize
\begin{minipage}[b]{15cm}
Mircea Cimpoea\c s, Simion Stoilow Institute of Mathematics, Research
unit 5, P.O.Box 1-764,\\
Bucharest 014700, Romania, E-mail: mircea.cimpoeas@imar.ro
\end{minipage}}
\end{document}